\documentclass[a4paper,12pt]{article}

\usepackage[utf8]{inputenc}
\usepackage{amstext}
\usepackage{amsfonts}
\usepackage{amsthm}
\usepackage{textcomp}
\usepackage{amssymb}
\usepackage{amscd}
\usepackage{amsmath}
\usepackage{xcolor}
\usepackage{tikz-cd}
\usepackage{caption}
\usepackage{enumitem}
\usepackage{theoremref}
\usepackage{mathtools}
\usepackage{float}
\setlist[enumerate]{label=({\arabic*})}

\newtheorem{defn}{Definition}[subsection]
\newtheorem{thm}[defn]{Theorem}

\newtheorem{lem}[defn]{Lemma}

\newtheorem{prop}[defn]{Proposition}

\newcommand{\Z}{\mathbb{Z}}

\date{\today}
\begin{document}
\title{On spaceability of shifts-like operators on $L^p$}

\author{Emma D'Aniello and Martina Maiuriello}

\newcommand{\Addresses}{{
  \bigskip
  \footnotesize
  
    E.~D'Aniello,\\
  \textsc{Dipartimento di Matematica e Fisica,\\ Universit\`a degli Studi della Campania ``Luigi Vanvitelli",\\
  Viale Lincoln n. 5, 81100 Caserta, ITALIA}\\
  \textit{E-mail address: \em emma.daniello@unicampania.it} \\
  
    M.~Maiuriello,\\
  \textsc{Dipartimento di Matematica e Fisica,\\ Universit\`a degli Studi della Campania ``Luigi Vanvitelli",\\
  Viale Lincoln n. 5, 81100 Caserta, ITALIA}\\
  \textit{E-mail address: \em martina.maiuriello@unicampania.it} 

}}

\maketitle

\begin{abstract}
We prove the spaceability of the set of hypercyclic vectors for {\em shifts-like operators}. Shift-like operators appear naturally as composition operators on $L^p(X)$, when the underlying space $X$ is dissipative. In the process of proving the main theorem, we provide, among other results of independent interest, a characterization of weakly mixing dissipative composition operators of bounded distortion.
\end{abstract}

\let\thefootnote\relax\footnote{\date{\today} 
}
%\newpage
\tableofcontents
\newpage

%##############################
\section{Introduction}
%##############################

\indent
During the last decades the study of hypercyclic operators on separable Banach spaces has developed into a very active research area. Its roots can be found in various areas of mathematics like, to name a few, Operator Theory, e.g., through the invariant subspace problem; Dynamical Systems, e.g., topological dynamics and ergodic theory; Semigroups of operators with applications to PDEs. In particular, hypercyclicity reveals its charm in Linear Dynamics, where it is one of the main ingredients in the most widely known definition of chaos, the chaos according to Devaney, and it is a natural strengthening of the weaker notion of chaos according to Li-Yorke. Hence, the investigations into hypercyclicity have shown, over the years, that this notion has been at the crossroads of several areas of mathematics, taking its examples and techniques from various domains. This central position has led, as a natural consequence, to the analysis of the vastness of the hypercyclic phenomenon or, in other words, of the ``size'' of the set of hypercyclic vectors. In this regards, it is already known, by the Birkhoff Transitivity Theorem, that hypercyclic vectors form a dense $G_{\delta}$-set, hence residual, meaning that hypercyclicity is a generic phenomenon \cite[Proposition 2.19]{GE}. More is known: given a hypercyclic operator on a certain separable Banach space, surprisingly, every vector of the space can be written as the sum of two different hypercyclic vectors \cite[Proposition 2.52]{GE}. Continuing this line of research, at the beginning of the 21st century, different concepts of largeness were presented: lineability, dense-lineability and spaceability appeared in the literature, meaning, respectively, existence of large vector spaces, of dense vector spaces and large closed vector spaces. Relations among the mentioned notions are completely understood: dense-lineability and spaceability are both stronger than lineability, while there is no relation between spaceability and dense-lineability \cite{aron2009dense}.
Motivated by the above, special attention has recently been devoted to these notions restricted to the set of hypercyclic vectors. In this sense, a number of results due to Herrero, Bourdon, Bes and Wengenroth evolved along the papers \cite{Herrero, Bourdon, Bes, Wengenroth}, proves that the set of hypercyclic vectors of a hypercyclic operator on an arbitrary vector space, is dense-lineable. In 1996, in the setting of Banach spaces, a criterium for spaceability, i.e., for the existence of closed subspaces, within the set of hypercyclic vectors is provided in \cite{Montes}. The monographs \cite{aronlineability} and \cite{GE} are exhaustive collections of recent results about hypercyclicity, (dense-)lineability and spaceability, in which several examples and applications can be found. \\
It is clear to anyone who handles a bit of Operator Theory and Linear Dynamics, that every new notion is first tested on weighted backward shifts. It turns out that the set of hypercyclic vectors of a weighted backward shift is always (dense-)lineable both in the bilateral and unilateral case \cite{Menet1, Bernal, Bourdon, Herrero} and, moreover, it is always spaceable for bilateral weighted shifts while, for the unilateral ones, characterizations are provided in \cite{Menet1}. The intimate relation, investigated in \cite{DAnielloDarjiMaiuriello2}, between weighted backward shifts and a large natural class of operators on $L^p$, $1 \leq p < \infty$, called ``shift-like operators'',  raises the natural question of how the set of hypercyclic vectors is structured in the case of the latter operators. Taking into account that shift-like operators appear naturally as dissipative composition operators, the above question naturally opens the door to the analysis of lineability and spaceability in the context of composition operators. \\
Motivated by the above, in this paper we settle the above question. It is organized as follows. Section 2 provides preliminary definitions and backgrounds. Section 3 attacks the above question in the dissipative context of bounded distortion, by providing a characterization of weakly mixing composition operators first, and analyzing the (dense-)lineability and spaceability within the set of hypercyclic vectors of composition operators, then.

%##############################
\section{Definitions and Background Results}
%##############################
As usual, ${\mathbb N}$ denotes the set of all positive integers and ${\mathbb N}_0={\mathbb N}\cup \{0\} .$ 
Recall that two linear operators $T:X \rightarrow X$ and $S:Y \rightarrow Y$ are said to be {\em topologically semi-conjugate} if there exists a linear, bounded, surjective map $\Pi: X \rightarrow Y$, called {\em factor map,} for which $\Pi \circ T=S\circ \Pi$. In such case, $S$ is called a {\em factor} of $T$. In particular, if $\Pi$ is a homeomorphism, $S$ and $T$ are said to be {\em topologically conjugate}.

%%%%%%%%%%%%%%%%%%%%%
\subsection{Weighted shifts and composition operators}
%%%%%%%%%%%%%%%%%%%%%

\begin{defn}[Linear composition operator] \label{compodyn}
Let $(X,{\mathcal B},\mu)$ be a $\sigma$-finite measure space and $f : X \to X$ be an injective bimeasurable transformation satisfying
\begin{equation}\label{condition}
 \exists c>0 \ \ : \ \  \mu(f^{-1}(B)) \leq c \mu(B) \ \textrm{ for every } B \in {\mathcal B}.
   \tag{$\star$}
\end{equation}
For $1 \le p <\infty$, the {\em composition operator induced by}\index{composition!operator} $f$ is the bounded linear operator 
\begin{align*}
    T_f \colon L^p(X,{\mathcal B},\mu) &\rightarrow L^p(X,{\mathcal B},\mu) \\
    \varphi &\mapsto \varphi \circ f . 
    \end{align*}
\end{defn}

In the sequel, we assume $f$ surjective and $f^{-1}$ satisfying ($\star$), implying that $T_f$ is invertible with $T_f^{-1}=T_{f^{-1}}$. The following two definitions are provided in \cite{DAnielloDarjiMaiuriello}.

\begin{defn}[Dissipative composition operators]
In the previous definition, if, in addition, there exists a wandering set $W \in \mathcal B$ such that 
\[X = \displaystyle{ {\dot \bigcup_{k \in \Z}} f^k (W)},\] where the symbol $\dot {\cup}$ denotes the disjoint union, then the composition operator $T_f$ is said to be {\em{dissipative}}.
\end{defn}

\begin{defn}[Bounded distortion]
Given a dissipative composition operator $T_f$, if in addition there exists $K>0$ such that
\begin{equation*}
 \dfrac{1}{K} \mu(f^k(W))\mu(B) \leq \mu(f^k (B))\mu (W) \leq K \mu(f^k(W))\mu(B), \tag{$\Diamond$}\label{eq:conditionbd}
\end{equation*}
for all $k \in \mathbb Z$ and  $B \in {\mathcal B}(W)$, where  ${\mathcal B}(W) =\{ B \cap W, B \in {\mathcal B} \},$ then $T_f$ is said to be a dissipative composition operator of {\em{bounded distortion}}.
\end{defn}

A special type of composition operators are the weighted shifts, the definition of which is recalled here.

\begin{defn} 
Let $A= {\mathbb Z}$ or $A = {\mathbb N}$. Let $X=\ell^p(A)$, $1 \leq p < \infty$ or $X=c_0(A).$ Let  $w=\{w_n\}_{n \in A}$ be a bounded sequence of scalars, called {\em weight sequence}.  Then, the {\em weighted backward shift $B_w$ on $X$} is defined by \[B_w(\{x_n\}_{n \in A}) =\{w_{n+1}x_{n+1}\}_{n \in  A}.\] If $A= {\mathbb Z}$, the shift is called {\em bilateral}. If $A = {\mathbb N}$, then the shift is {\em unilateral}. 
\end{defn}

It is proved in \cite{DAnielloDarjiMaiuriello} that the relation between composition operators and weighted shifts is deeper in the dissipative setting with bounded distortion, as the following result shows.

\begin{lem} \cite[Lemma 4.2.3]{DAnielloDarjiMaiuriello} \label{factorBw}
Let $T_f$ be a dissipative composition operator of bounded distortion, generated by a wandering set $W$. Consider the weighted backward shift $B_w$ on $\ell^p(\Z)$ with weights
\[w_{k} =  \left( \frac{\mu(f^{k-1}(W))}{\mu(f^{k}(W))}\right)^{\frac{1}{p}} . \] 
Then, $B_w$ is a factor of $T_f$ by the factor map $\Pi: L^p(X) \rightarrow \ell ^p({\mathbb Z}) $ defined as $\Pi(\varphi)= \{x_{k}\}_{k \in {\mathbb Z}}, $ where 
 \[x_{k} = \dfrac{\mu(f^{k}(W))^{\frac{1}{p}}}{\mu(W)} \int_{W}  \varphi \circ f^{k} d \mu, \ \ k \in \Z.  \]
\end{lem}

In \cite{DAnielloDarjiMaiuriello2}, using the above result, it is proved that dissipative composition operators of bounded distortion behave as bilateral weighted shifts. \\
For this reason, the authors call the dissipative composition operators of bounded distortion, {\em{shifts-like operators}.} For a detailed presentation of the above, and related, topics see \cite{Maiuriello}, which is also a good source of updated literature.\\
 In the sequel, the spaces $L^p(X,{\mathcal B},\mu)$ will be simply denoted by $L^p(X)$, $1 \leq p < \infty$.

%@@@@@@@@@@@@@@@@@@
\subsection{Hypercyclicity and weak mixing}
%@@@@@@@@@@@@@@@@@@
Given a complex Banach space $X$, from now on $T:X \rightarrow X$ denotes a bounded linear operator on $X$, and the space $X$ is always assumed to be separable.
 
\begin{defn}
The operator $T$ is said to be
\begin{itemize}
\item{\em hypercyclic} if it admits a {\em hypercyclic vector,} i.e., if there exists $x \in X$ such that $Orb(x,T)$ is dense in $X$;
\item{\em weakly mixing} if the operator $T\oplus T : (x,y) \in X \times X \mapsto (Tx,Ty) \in X \times X$ is hypercyclic.
\end{itemize}
\end{defn}

\noindent
In the sequel, $HC(T)$ denotes the set of hypercyclic vectors of the operator $T$.
There are no hypercyclic operators on finite-dimensional spaces \cite{Kitai, Rolewicz}. The set $HC(T)$ is residual in $X$, meaning that hypercyclicity is a generic phenomenon \cite[Theorem 2.19]{GE}. Weak mixing implies hypercyclicity, but the implication cannot be reverted.

\noindent
The next fundamental result, known as the Gethner–Shapiro Criterion, will be applied in the sequel.

\begin{thm}\label{GS} \cite[Theorem 3.10]{GE} Given the operator $T$, if there are dense subsets $X_0,Y_0 \subset X$, an increasing sequence $\{n_k\}_{k \geq 1}$ of positive integers, and $S:Y_0 \rightarrow Y_0$ such that, for any $x\in X_0$, $y\in Y_0$,
\begin{itemize}
\item[(a)] $T^{n_k}x\rightarrow 0$
\item[(b)] $S^{n_k}y \rightarrow 0$
\item[(c)] $TSy=y$
\end{itemize}
then $T$ is weakly mixing.
\end{thm}

%@@@@@@@@@@@@@@@@@@
\subsection{Lineability and spaceability}
%@@@@@@@@@@@@@@@@@@
The following concepts of lineability and spaceability were introduced by Aron, Gurariy and Seoane Sepulveda in 2005 \cite{aron2005lineability}, and, since then, they have received a lot of attention and have been widely investigated.

\begin{defn}
Let $M \subseteq X$ and $\mu$ be a cardinal number. The set $M$ is said to be
\begin{itemize}
\item{\em{$\mu$-lineable}} if $M\cup \{0\}$ contains a vector space of dimension $\mu$;
\item{\em{$\mu$-spaceable}} if $M\cup \{0\}$ contains a closed vector space of dimension $\mu$;
\item{\em{$\mu$-dense-lineable}} if $M\cup \{0\}$ contains a dense vector space of dimension $\mu$.
\end{itemize}
\end{defn}

We refer to the set $M$ as {\em lineable, spaceable, dense-lineable,} if the respective existing subspace is infinite dimensional. The following diagram, showing the basic relations between the mentioned notions, cannot be improved, in the sense that none of the implications can be reverted \cite[Theorem 7.2.1]{aronlineability}. 

\begin{figure}[H]
\centering
\begin{tikzcd}[sep=small,arrows=Rightarrow]
& \text{{Spaceability}}\arrow[dr]  & &  \text{{Dense-lineability}}\arrow[dl]  &\\
 &   &  \text{Lineability} &  & \\
\end{tikzcd}
\end{figure}
The following propositions concern the relevant set of vectors $HC(T)$.

\begin{prop} \label{prop1}
If the operator $T$ is hypercyclic, then the set $HC(T)$ of hypercyclic vectors of $T$ is dense-lineable and, moreover, the dense subspace obtained is $T-$invariant \cite{Bernal, Bourdon, Herrero}.
\end{prop}

Let $\rm{Cof}=\{E \subseteq X : E \text{ subspace of finite codimension}\}$.  The following necessary and sufficient condition, involving $\rm{Cof}$, for $HC(T)$ to be spaceable is provided in \cite{Menet1}:

\begin{prop} \label{prop2} \cite[Corollary 2.1]{Menet1}
Let $T$ be a weakly mixing operator. The set $HC(T)$ is spaceable if and only if \[ \sup_{n \geq 1} \sup_{E \in \rm{Cof}} \inf_{x \in E\setminus \{0\}} \dfrac{\Vert T^nx \Vert}{\Vert x \Vert} \leq 1.\]
\end{prop}

%%%%%%%%%%%%%%%%%%%%%%%%%%%%%%
\subsubsection{Lineability and spaceability of $HC(B_w)$}
As it is known, the set $HC(B_w)$ of a hypercyclic bilateral (unilateral) weighted shift $B_w$ is dense-lineable and, therefore, lineable (see (1) of Proposition \ref{prop1}). 
The next result investigates the spaceability for unilateral and bilateral weighted shifts.

\begin{thm} \cite[Theorem 3.3, Theorem 3.5]{Menet1} 
The followings hold.
\begin{enumerate}
\item The set $HC(B_w)$ of a hypercyclic unilateral weighted shift $B_w$ is spaceable if and only if \[ \sup_{n \geq 1} \inf_{k \geq 1} \prod_{\nu=1}^{n} \vert w_{\nu+k} \vert < \infty. \]
\item The set $HC(B_w)$ of a hypercyclic bilateral weighted shift $B_w$ is spaceable.
\end{enumerate}
\end{thm}

%%%%%%%%%%%%%%%%%%%%%%%%%%%%%%%%%%%%%%%%%%%%%%%%%%%%
\section{Lineability and spaceability {\bf of $HC(T_{f})$}}
%%%%%%%%%%%%%%%%%%%%%%%%%%%%%%%%%%%%%%%%%%%%%%%%%%%%
As for $HC(B_w)$, the set $HC(T_f)$ is dense-lineable and, therefore, lineable (see (1) of Proposition \ref{prop1}). The main result of the paper is the characterization of spaceability of $HC(T_f)$. In order to prove it, we need to investigate the role of weak mixing in the context of composition operators.
%%%%%%%%%%%%%%%%%%%
\subsection{Weak mixing for $T_f$}

\begin{thm} \label{weakmixTf}
Let $T_f$ be a dissipative composition operator of bounded distortion, generated by a wandering set $W$.
The followings are equivalent:
\begin{itemize}
\item[(i)] $T_f$ is hypercyclic;
\item[(ii)] $T_f$ is weakly mixing;
\item[(iii)] $\displaystyle{\exists \{n_k\}_{k \geq 1} :  \forall j \in \mathbb Z, \lim_{k \rightarrow + \infty } \dfrac{\mu(f^{j-n_k}(W))}{\mu(f^j(W))}=0 \ \text{and} \ \ \lim_{k \rightarrow + \infty } \dfrac{\mu(f^{j}(W))}{\mu(f^{j+n_k}(W))}=\infty.}$
\end{itemize}
\end{thm}

\begin{proof}
$(i)\Rightarrow (iii)$.\\
By \cite[Theorem M]{DAnielloDarjiMaiuriello2}, $T_f$ is hypercyclic if and only if the associated weighted backward shift $B_w$, that is the one with weights $w_{h} =  \left( \frac{\mu(f^{h-1}(W))}{\mu(f^{h}(W))}\right)^{\frac{1}{p}}$, is hypercyclic. By \cite[page 102]{GE}, as $B_w$ is hypercyclic, then
\[ \exists \{n_k\}_{k \geq 1} :  \forall j \in \mathbb Z, \lim_{k \rightarrow + \infty } \prod_{h=j-n_k+1}^j w_h =0 \ \& \ \ \lim_{k \rightarrow + \infty } \prod_{h=j+1}^{j+n_k} \vert w_h\vert =\infty.\]
Hence, $(iii)$ follows by substituting $w_{h} =  \left( \frac{\mu(f^{h-1}(W))}{\mu(f^{h}(W))}\right)^{\frac{1}{p}}$ in the above formula.\\
\noindent
$(iii)\Rightarrow (ii)$.\\
We now show that conditions $(a), (b), (c)$ of Theorem \ref{GS} are satisfied with $T=T_f$ and $S=T_{f^{-1}}$, implying that $T_f$ is weakly mixing.
By assumption, the composition operator $T_f$ is invertible with $T_f^{-1}=T_{f^{-1}}.$ Condition $(c)$ is trivially satisfied.\\
 In order to show that also $(a)$ and $(b)$ hold, in accordance with the notation of Theorem \ref{GS}, we let  
\[X_0=Y_0=\left \{\sum_i a_i \chi_{B_i} : a_i \in \mathbb C, B_i \subseteq f^{j_i}(W), j_i \in \mathbb Z, \mu(B_i)>0, B_i \cap B_{i'} = \emptyset, i \neq i' \right \},\]
 which is dense in $L^p(X)$. 
For each $\varphi \in X_0$,
\begin{eqnarray*}
\Vert T_f^{n_k} \varphi \Vert_p^p&=&\Vert \sum_i a_i \chi_{B_i} \circ f^{n_k} \Vert_p^p\\
&=&\sum_i \vert a_i \vert^p \int_X \vert \chi_{B_i} \circ f^{n_k} \vert^p d\mu\\
&=& \sum_i \vert a_i \vert^p \mu(f^{-n_k}(B_i)) \leq  \sum_i \vert a_i \vert^p \mu(f^{j_i-n_k}(W))\\
&=& \sum_i \vert a_i \vert^p \mu(f^{-n_k}(B_i)) \leq  \sum_i \vert a_i \vert^p \dfrac{\mu(f^{j_i-n_k}(W))}{f^{j_i}(W)}f^{j_i}(W)
\end{eqnarray*}
and
\begin{eqnarray*}
\Vert T_{f^{-1}}^{n_k} \varphi \Vert_p^p&=&\Vert \sum_i a_i \chi_{B_i} \circ f^{- n_k} \Vert_p^p\\
&=&\sum_i \vert a_i \vert^p \int_X \vert \chi_{B_i} \circ f^{- n_k} \vert^p d\mu\\
&=& \sum_i \vert a_i \vert^p \mu(f^{n_k}(B_i)) \leq  \sum_i \vert a_i \vert^p \mu(f^{j_i+n_k}(W))\\
&=& \sum_i \vert a_i \vert^p \mu(f^{n_k}(B_i)) \leq  \sum_i \vert a_i \vert^p \dfrac{\mu(f^{j_i+n_k}(W))}{f^{j_i}(W)}f^{j_i}(W).
\end{eqnarray*}
Hence, $(iii)$ implies both $\Vert T_f^{n_k} \varphi \Vert_p\rightarrow 0$ and $\Vert T_{f^{-1}}^{n_k} \varphi \Vert_p\rightarrow 0$ as $k \rightarrow + \infty$, meaning that $(a)$ and $(b)$ of Gethner-Shapiro Criterion are satisfied, and hence $T_f$ is weakly mixing, i.e., $(ii)$ holds.\\
\noindent
$(ii)\Rightarrow (i)$. \\
This is obvious.
\end{proof}

As the weighted backward shifts are properly contained in the class of composition operators, the previous result generalizes the equivalence between weak mixing and hypercyclicity (proved for shifts in  \cite[Proposition 4.13-(a)]{GE}) to this more general class.  As a consequence of the above theorem together with Lemma \ref{factorBw}, it turns out, as the next theorem shows, that dissipative composition operators of bounded distortion have a {\em{shifts-like behavior}} concerning weak mixing, extending so the result given in \cite[Theorem M]{DAnielloDarjiMaiuriello2}.

\begin{thm} \label{Tf}
Let $T_f$ be a dissipative composition operator of bounded distortion, generated by a wandering set $W$. Consider the weighted backward shift $B_w$ as in Lemma \ref{factorBw}.
Then, $T_f$ is weakly mixing if and only if $B_w$ is. 
\end{thm}
\begin{proof}
As weak mixing and hypercyclicity are equivalent notions for $T_f$ (Theorem \ref{weakmixTf}) and $B_w$ (\cite[Proposition 4.13-(a)]{GE}), then the thesis follows by applying  \cite[Theorem M]{DAnielloDarjiMaiuriello2}.
\end{proof}

Now, we have almost all the tools to investigate $HC(T_f)$.

%%%%%%%%%%%%%%%%%%%%%%%%%%%%%
\subsection{Spaceability of $HC(T_f)$}
 In order to analyze the spaceability of  $HC(T_f)$, we need some auxiliary results. 

\begin{lem} \label{conditionmix}
Let $T_f$ be a dissipative composition operator, generated by a wandering set $W$. Consider the following statements:
\begin{itemize}
\item[(i)] $\displaystyle{ \sup_{n \geq 1} \inf_{k \in \mathbb Z}\left ( \dfrac{\mu(f^k(W))}{\mu(f^{k+n}(W))} \right )^{\frac{1}{p}} \leq 1,}$
\item[(ii)] $\displaystyle{ \sup_{n \geq 1} \sup_{E \in \rm{Cof}} \inf_{\varphi \in E\setminus \{0\}} \dfrac{\Vert T_f^n \varphi \Vert_p}{\Vert \varphi \Vert_p} \leq 1.}$
\end{itemize}
Then $(i)$ implies $(ii)$.
\end{lem}

\begin{proof}
%$(i) \Rightarrow (ii)$\\
By hypothesis, for each $n \geq 1 $ there exists $k \in \mathbb Z$ such that \[\left (\dfrac{\mu(f^k(W))}{\mu(f^{k+n}(W))}  \right )^{\frac{1}{p}} \leq 1. \tag{$\bullet$}\]
Let $E \in \rm{Cof}$, i.e., $E$ is a cofinite subspace of $L^p(X)$, meaning that $L^p(X)=E \oplus F$ where $\rm{dim}(F)=m < \infty$. Hence, for each $\varphi_1, \varphi_2, ..., \varphi_{m+1} \in E$ linearly independent, there exist $a_1, a_2, ..., a_{m+1}$ scalars such that $a_1 \varphi_1 + ...+ a_{m+1}\varphi_{m+1} \in E \setminus \{0\}$.
Fix any $k_{1}$, $\dots$, $k_{m+1}$ in $\mathbb Z$, all distinct, and take $a_1, a_2, ..., a_{m+1}$ such that
 \[\varphi= a_1 \chi_{f^{k_1}(W)} + a_2 \chi_{f^{k_2}(W)}+...+ a_{m+1} \chi_{f^{k_{m+1}}(W)} \in E \setminus \{0\},\]
where, by definition of a dissipative composition operator, the $f^{k_i}(W)$'s are all pairwise disjoint.\\
Note that, for each $n \geq 1$,
\begin{eqnarray*}
\Vert T_f^n \varphi \Vert_p^p&=& \Vert (a_1 \chi_{f^{k_1}(W)} +...+ a_{m+1} \chi_{f^{k_{m+1}}(W)}) \circ f^n \Vert_p^p \\
&=& \vert a_1\vert^p \int_X\vert  \chi_{f^{k_1}(W)} \circ f^n \vert^p d\mu + ...+ \vert a_{m+1} \vert^p \int_X \vert \chi_{f^{k_{m+1}}(W)} \circ f^n \vert ^p d\mu\\ 
&=& \vert a_1\vert^p \mu({f^{k_1-n}(W)}) + ...+ \vert a_{m+1} \vert^p \mu({f^{k_{m+1}-n}(W)})   \ \ \ (\text{apply } (\bullet))\\
&\leq & \vert a_1\vert^p \mu({f^{k_1}(W)}) + ...+ \vert a_{m+1} \vert^p \mu({f^{k_{m+1}}(W)}) \\
&=& \Vert \varphi \Vert _p^p.
\end{eqnarray*}
Hence, \[\dfrac{\Vert T_f^n\varphi \Vert_p}{\Vert \varphi \Vert _p} \leq 1, \forall n \geq 1 .\]
The arbitrariness of $E \in \rm{Cof}$ implies that 
\[ \sup_{n \geq 1} \sup_{E \in \rm{Cof}} \inf_{\varphi \in E\setminus \{0\}} \dfrac{\Vert T_f^n \varphi \Vert_p}{\Vert \varphi \Vert_p} \leq 1.\]
That is, $(ii)$ holds.
\end{proof}

\begin{thm}
The set $HC(T_f)$ of a hypercyclic dissipative composition operator of bounded distortion $T_f$ is spaceable.
\end{thm}

\begin{proof}
Let $T_f$ be hypercyclic. This is equivalent, by Theorem \ref{weakmixTf}, to the existence of an increasing sequence $\{n_k\}_{k \geq 1}$ such that for all $j \in \mathbb Z,$ 
\[ \lim_{k \rightarrow + \infty } \dfrac{\mu(f^{j-n_k}(W))}{\mu(f^j(W))}=0. \tag{$\spadesuit$}\]
Assume, by contradiction, that $HC(T_f)$ is not spaceable. By Proposition \ref{prop2}, it follows that
\[ \sup_{n \in \mathbb N} \sup_{E \in \rm{Cof}} \inf_{\varphi \in E \setminus \{0\}} \dfrac{\Vert T_f^n \varphi \Vert_p}{\Vert \varphi \Vert_p} >1\] 
implying, by Lemma \ref{conditionmix}, that
\[ \sup_{n \geq 1} \inf_{k \in \mathbb Z}\left ( \dfrac{\mu(f^k(W))}{\mu(f^{k+n}(W))} \right )^{\frac{1}{p}} > 1.\]
Hence, there exist $n \geq 1, C>1$ such that, for each $k \in \mathbb Z$, 
\[\dfrac{\mu(f^k(W))}{\mu(f^{k+n}(W))}>C^p>1.\]
Note that, by condition $(\star)$, $\dfrac{\mu(f^k(W))}{\mu(W)} \leq c^{k}$ for each $k \in \mathbb Z$.\\
Writing $n_k=m_kn + q$, with $0 \leq q \leq n-1$,  then
\begin{eqnarray*}
\dfrac{\mu(f^{j-n_k}(W))}{\mu(f^j(W))} &=& \dfrac{\mu(f^{j-m_kn-q}(W))}{\mu(f^j(W))}\\
&=&\dfrac{\prod_{h=1}^{m_k} \prod_{\nu=1}^n \frac{\mu(f^{j-hn+(\nu-1)}(W))}{\mu(f^{j-hn+\nu}(W))} }{\prod_{l=1}^q\frac{\mu(f^{j-m_kn-l+1}(W))}{\mu(f^{j-m_kn-l}(W))}}\\
&= &\left[ \dfrac{\mu(f^{j-m_kn-q}(W))}{\mu(f^{j-m_kn-q+1}(W))} \cdot  \dfrac{\mu(f^{j-m_kn-q+1}(W))}{\mu(f^{j-m_kn-q+2}(W))} \cdot \ldots \cdot \dfrac{\mu(f^{j-m_kn-q + q -1}(W))}{\mu(f^{j-m_kn - q + q}(W))}\right]\\
& \ \  & \cdot \left[\dfrac{\mu(f^{j-m_kn}(W))}{\mu(f^{j-m_kn+1}(W))} \cdot  \ldots \cdot  \dfrac{\mu(f^{j-m_kn+n-1}(W))}{\mu(f^{j-m_kn+n}(W))}\right] \\
& \ \  & \cdot \dfrac{\mu(f^{j-m_kn+n}(W))}{\mu(f^{j-m_kn+n+1}(W))}\cdot \ldots \cdot  \dfrac{\mu(f^{j-1}(W))}{\mu(f^{j}(W))} \\
&\geq & \dfrac{1}{c^{n+q}} \cdot  \dfrac{\mu(f^{j-m_kn+n}(W))}{\mu(f^{j-m_kn+n+1}(W))} \cdot \ldots \cdot  \dfrac{\mu(f^{j-1}(W))}{\mu(f^{j}(W))} \\
&=& \dfrac{1}{c^{n+q}} \cdot  \dfrac{\mu(f^{j-m_kn+n}(W))}{\mu(f^{j-m_kn+2n}(W))} \cdot \dfrac{\mu(f^{j-m_kn+2n}(W))}{\mu(f^{j-m_kn+3n}(W))}  \ldots \cdot  \dfrac{\mu(f^{j-m_kn+(m_k-1)n}(W))}{\mu(f^{j-m_kn+m_kn}(W))} \\
&=& \dfrac{1}{c^{n+q}} \cdot  \dfrac{\mu(f^{j-m_kn+n}(W))}{\mu(f^{j-m_kn+2n}(W))} \cdot \dfrac{\mu(f^{j-m_kn+2n}(W))}{\mu(f^{j-m_kn+3n}(W))}  \ldots \cdot  \dfrac{\mu(f^{j-n}(W))}{\mu(f^{j}(W))} \\
&\geq & \dfrac{1}{c^{n+q}} C^{p(m_k+1)}.
\end{eqnarray*}
This is in contradiction with $(\spadesuit)$, thus implying that $HC(T_f)$ has to be spaceable.
\end{proof}

\section*{Acknowledgement}
This research has been partially supported by the project ``Vain-Hopes'' within the program Valere of Università degli Studi della Campania ``Luigi Vanvitelli''. It has been partially accomplished within the UMI Group TAA ``Approximation Theory and Applications''.

\bibliographystyle{siam}
\bibliography{biblio}

\Addresses

\end{document}